\documentclass[12pt,reqno]{amsart}
\usepackage{amsmath,a4wide}
\usepackage{xfrac}
\usepackage{stmaryrd,mathrsfs,bm,amsthm,mathtools,yfonts,amssymb,color}
\usepackage{xcolor}
\usepackage{courier}

\begingroup
\newtheorem{theorem}{Theorem}[section]
\newtheorem{lemma}[theorem]{Lemma}
\newtheorem{proposition}[theorem]{Proposition}

\endgroup

\theoremstyle{definition}
\newtheorem{definition}[theorem]{Definition}

\newtheorem{ipotesi}[theorem]{Assumption}

\numberwithin{equation}{section}
\numberwithin{subsection}{section}

\newcommand\supp{{\rm spt}}
\newcommand{\sing}{{\rm Sing}}
\newcommand{\reg}{{\rm Reg}}

\newcommand\res{\mathop{\hbox{\vrule height 7pt width .3pt depth 0pt
\vrule height .3pt width 5pt depth 0pt}}\nolimits}

\newcommand{\gr}{{\rm Gr}}
\newcommand{\bT}{\mathbf{T}}
\newcommand{\bG}{\mathbf{G}}
\newcommand{\cH}{{\mathcal{H}}}
\newcommand{\bJ}{{\mathbf{J}}}

\newcommand{\p}{{\mathbf{p}}}

\newcommand{\bI}{{\mathbf{I}}}
\newcommand{\bE}{{\mathbf{E}}}

\newcommand{\bOmega}{{\mathbf{\Omega}}}

\newcommand{\bB}{{\mathbf{B}}}
\newcommand{\bC}{{\mathbf{C}}}

\newcommand{\be}{{\mathbf{e}}}
\newcommand{\bd}{{\mathbf{d}}}

\newcommand\N{{\mathbb N}}

\newcommand\R{{\mathbb R}}

\newcommand{\eps}{{\varepsilon}}
\newcommand{\bA}{\mathbf{A}}
\newcommand{\bmax}{\mathbf{m}}

\def\XXint#1#2#3{{\setbox0=\hbox{$#1{#2#3}{\int}$ }
\vcenter{\hbox{$#2#3$ }}\kern-.6\wd0}}

\newcommand{\cone}{{\times\hspace{-0.6em}\times\,}}
\newcommand{\Lip}{{\rm {Lip}}}

\newcommand{\cG}{{\mathcal{G}}}

\newcommand{\cN}{{\mathcal{N}}}

\newcommand{\cT}{{\mathcal{T}}}

\newcommand{\cQ}{{\mathcal{Q}}}

\newcommand{\mass}{{\mathbf{M}}}
\renewcommand\d{\mathbf{d}}
\newcommand\e{\mathbf{e}}

\def\I#1{{\mathcal{A}}_{#1}}

\newcommand{\Iqs}{{\mathcal{A}}_Q(\R^{n})}
\newcommand{\Iq}{{\mathcal{A}}_Q}
\def\a#1{\left\llbracket{#1}\right\rrbracket}

\newcommand{\norm}[2]{\left\|#1\right\|_{#2}}

\newcommand{\D}{\textup{Dir}}
\newcommand{\de}{\partial}
\newcommand{\xii}{{\bm{\xi}}}
\newcommand{\ro}{{\bm{\rho}}}
\newcommand{\g}{{g'}}

\newcommand{\etaa}{{\bm{\eta}}}
\newcommand{\ph}{\varphi}

\newcommand{\lin}{{\text{lin}}}

\newcommand{\bh}{\mathbf{h}}

\newcommand{\Ema}{\color{black}}

\title[Lipschitz approximation for almost minimal current]
{Regularity theory for $2$-dimensional almost minimal currents I: Lipschitz approximation}
\author{Camillo De Lellis, Emanuele Spadaro and Luca Spolaor}

\begin{document}

\begin{abstract}
We construct Lipschitz $Q$-valued functions which approximate carefully integral currents when their cylindrical 
excess is small and they are almost minimizing in a suitable sense. This result is used in two subsequent works to prove the
discreteness of the singular set for the following three classes of $2$-dimensional integral currents: 
area minimizing in Riemannian manifolds, semicalibrated and spherical cross sections of $3$-dimensional
area minimizing cones.
\end{abstract}

\maketitle

This paper is the second in a series of works aimed at establishing an optimal regularity theory
for $2$-dimensional integral currents which are almost minimizing in a suitable sense. Building upon the monumental
work of Almgren \cite{Alm}, Chang in \cite{Chang} established that $2$-dimensional area-minimizing currents in
Riemannian manifolds are classical minimal surfaces, namely they are regular (in the interior) except for a discrete
set of branching singularities. The argument of Chang is however not entirely complete since a key starting point 
of his analysis, the existence of the so-called ``branched center manifold'', is only sketched in the appendix of
\cite{Chang} and requires the understanding (and a suitable modification) of the most involved portion of the monograph \cite{Alm}. 

An alternative proof of Chang's theorem has been found by Rivi\`ere and Tian in \cite{RT1} for the special case
of $J$-holomorphic curves. Later on the approach of Rivi\`ere and Tian has been generalized by Bellettini and Rivi\`ere
in \cite{BeRi} to handle a case which is not covered by \cite{Chang}, namely that of special Legendrian cycles in
$\mathbb S^5$ (see also \cite{Be3} for a further generalization).

Meanwhile the first and second author revisited Almgren's theory giving a much shorter version of his program
for proving that area-minimizing currents are regular up to a set of Hausdorff codimension $2$, cf. \cite{DS1,DS2,DS3,DS4,DS5}.
In this note and its companion papers \cite{DSS3,DSS4} we build upon the latter works in order to give a
complete regularity theory which includes both the theorems of Chang and Bellettini-Rivi\`ere as special cases{\Ema , in particular recovering the fine description of the structure of singular points proven by Chang and extending this picture to the cases of semicalibrated currents and spherical cross-sections (we refer to \cite{DSS3, DSS4} for more precise statements)}. 

We start by introducing the following terminology (cf. \cite[Definition 0.3]{DSS1}).

\begin{definition}\label{d:semicalibrated}
Let $\Sigma \subset R^{m+n}$ be a $C^2$ submanifold and $U\subset \R^{m+n}$ an open set.
\begin{itemize}
  \item[(a)] An $m$-dimensional integral current $T$ with finite mass and $\supp (T)\subset \Sigma\cap U$ is area-minimizing in $\Sigma\cap U$
if $\mass(T + \partial S)\geq \mass(T)$ for any $(m+1)$-dimensional integral current $S$ with $\supp (S) \subset \subset \Sigma\cap U$.
 \item[(b)] A semicalibration (in $\Sigma$) is a $C^1$ $m$-form $\omega$ on $\Sigma$ such that 
  $\|\omega_x\|_c \leq 1$ at every $x\in \Sigma$, where $\|\cdot \|_c$ denotes the comass norm on $\Lambda^m T_x \Sigma$. 
  An $m$-dimensional integral current $T$ with $\supp (T)\subset \Sigma$ is {\em semicalibrated} by $\omega$ if $\omega_x (\vec{T}) = 1$ for $\|T\|$-a.e. $x$.
  \item[(c)]  An $m$-dimensional integral current $T$ supported in $\partial \bB_R (p) \subset \R^{m+n}$ is a {\em spherical cross-section of an area-minimizing cone} if ${p\cone T}$ is area-minimizing. 
\end{itemize}
\end{definition}

Calibrated submanifolds, namely currents $T$ as in (b) where the calibrated form is closed, have been central objects of study in 
several areas of differential geometry and mathematical physics since the seminal work of Harvey and Lawson, cf. \cite{HL}. Two primary examples are holomorphic subvarieties and special Lagrangians in Calabi-Yau manifolds, which play an important role in string theory (especially regarding mirror symmetry, cf. \cite{Joyce, SYZ}) but also emerge naturally in gauge theory (see \cite{Tian}). Semicalibrations are a natural generalization of calibrations: since the condition $d\omega =0$ on the calibrating form is rather rigid and in particular very unstable under deformations. In fact semicalibrations were considered already in \cite{Tian} (cf. Section 6 therein) and around the same time they became rather popular in string theory, when several authors directed their attention to non-Calabi-Yau manifolds (the subject is nowadays known as ``flux compactification'', cf. \cite{Grana}): in that context the natural notion to consider is indeed a special class of semicalibrating forms (see for
  instance the works \cite{GPT,Gutowski}, where these are called quasi calibrations).

\medskip

In what follows, given an integer rectifiable current $T$, we denote by $\reg (T)$ the subset of $\supp (T)\setminus \supp (\partial T)$ consisting of those points $x$ for which there is a neighborhood $U$ such that $T\res U$ is a (constant multiple of) a regular oriented submanifold. Correspondingly, $\sing (T)$ is
the set $\supp (T)\setminus (\supp (\partial T)\cup \reg (T))$. Observe that $\reg (T)$ is relatively open in $\supp (T) \setminus \supp (\partial T)$
and thus $\sing (T)$ is relatively closed.
The main result of this and the works \cite{DSS3,DSS4} is then the following

\begin{theorem}\label{t:finale}
Let $\Sigma$ and $\omega$ be as in Definition \ref{d:semicalibrated}, let $T$ be as in (a), (b) or (c) and assume in addition
that $m=2$, that $\Sigma$ is of class $C^{3,\alpha}$ and $\omega$ of class $C^{2,\alpha}$ for some positive $\alpha$. Then $\sing (T)$
is discrete.
\end{theorem}

Clearly Chang's result is covered by case (a). As already pointed out, the proof of Theorem~\ref{t:finale} gives in fact more information, namely an accurate description of the behavior of $T$ around any singular point. This is the exact analog of the singularity description provided by Chang \cite{Chang} for the area minimizing case, whose validity is therefore extended to both cases (b) and (c) of Definition \ref{d:semicalibrated}.
The results of Theorem~\ref{t:finale} are
optimal, because of the well-known examples of area minimizing currents
induced by singular complex curves. 
Note, however, that there are many singular semicalibrated
currents which are not calibrated, and
we give an example in the appendix.

The program of extending the Almgren-Chang regularity theory to general semicalibrated currents was started by Rivi\`ere and Tian in \cite{RT2} and indeed their alternative proof of Chang's theorem was meant as a first step towards case (b) of Theorem \ref{t:finale} (cf. \cite[page 743]{RT2}). The first notable contribution which goes beyond the Almgren-Chang result is due to Pumberger and Rivi\`ere in \cite{PuRi}, and important groundbreaking results were then achieved by Bellettini and Rivi\`ere in \cite{BeRi} and by Bellettini in \cite{Be3}. In particular
\cite{BeRi} proved the theorem above for {\Ema Legendrian} cycles in ${\mathbb S}^5$, which form a special subclass of both (b) and (c). The result and the methods were then extended in \cite{Be3} to
a class of $2$-dimensional semicalibrated currents in 5-dimensional manifolds which, roughly speaking, are based on Legendrian cycles as local models. 
In this and the notes \cite{DSS3,DSS4} we give a complete answer in the general $2$-dimensional case. In higher dimensions Almgren's famous bound on the Hausdorff dimension of the singular set has been extended to semicalibrated currents by the third author in \cite{Luca}.

\medskip

Following the Almgren-Chang program, Theorem \ref{t:finale} will be established through a suitable ``blow-up argument'' which requires four essential tools:
\begin{itemize}
\item[(i)] The uniqueness of tangent cones for $T$. This result is a, by now classical, theorem of White for area-minimizing $2$-dimensional currents in the Euclidean
space, cf. \cite{Wh}. Chang extended it to case (a) in the appendix of \cite{Chang}, whereas Pumberger and Rivi\`ere covered case (b) in \cite{PuRi}. 
A general derivation of these results for a wide class of almost minimizers has been given in \cite{DSS1}: the theorems in there
cover, in particular, all the cases of Definition \ref{d:semicalibrated}.
\item[(ii)] The theory of multiple-valued functions, pioneered by Almgren in \cite{Alm}, for which we
will use the results and terminology of the papers \cite{DS1,DS2}.
\item[(iii)] A suitable approximation procedure for integer rectifiable
currents with graphs of multiple valued functions. The one needed in case (a) is already contained in \cite{DS3}, but the latter reference
does not cover the cases (b) and (c): the purpose of this note is to extend the theorems in
\cite{DS3} to these cases.
\item[(iv)] The so-called ``center manifold'': this will be constructed in \cite{DSS3}, whereas the final argument for
Theorem \ref{t:finale} will then be given in \cite{DSS4}.
\end{itemize}

In fact this note does more than just providing (iii) for the cases of (b) and (c), because we give an approximation theorem for almost minimal currents in any dimension $m$, see Definition~\ref{d:Omega-minimal} for the precise condition. 
Indeed, relaxing the minimizing condition in the regularity theory is a central theme in geometric measure theory: from the one hand it could be the first step towards the analysis of different elliptic functionals, from the other hand it has many applications in a variety of problems in which the minimizing condition must be weakened (the examples are numerous: we just cite the fundamental work of Almgren on elliptic variational problems with constraints \cite{Alm-memoirs} and the celebrated paper of Schoen and Simon on stable hypersurfaces, \cite{SS}). However, there are very few results in this direction in higher codimension: compared to the codimension one case the task is much harder, since several delicate arguments of the Almgren-Chang depend sensibly upon the minimizing assumption. This note gives a first contribution by establishing a strong approximation theorem under a very natural condition, see Definition \ref{d:Omega-minimal} for the precise formulation.

\subsection{Acknowledgments}  The research of Camillo De Lellis and Luca Spolaor has been supported by the ERC grant RAM (Regularity for Area Minimizing currents), ERC 306247.

\section{Notation and statement of the main theorem}

{\Ema We introduce the notion of almost minimizers that we are going to use in the paper.

\begin{definition}[$\bOmega$-minimality]\label{d:Omega-minimal}
Let $\bOmega$ be a positive constant.
An integer rectifiable $m$-dimensional current with compact support in $\R^{m+n}$ is called {\em $\bOmega$-minimal} if
\begin{equation}\label{e:Omega}
\mass (T) \leq \mass (T + \partial S) + \bOmega\, \mass (S)\qquad \forall S\in {\bf I}_{m+1} (\R^{m+n})\quad \mbox{with compact support.}
\end{equation}
\end{definition}

In order to state the main result, we need to introduce some notation.}
With $\bB_r (p)$ and $B_r (x)$ we denote, respectively, the open ball with radius $r$ and center $p$ in $\mathbb R^{m+n}$ and the open ball with radius $r$ and center $x$ in $\mathbb R^m$.  $\bC_r (x)$ will always denote the cylinder $B_r (x)\times \mathbb R^n$ and the point $x$ will be omitted when it is the origin. In fact, by a slight abuse of notation, we will often treat the center $x$ as a point in $\mathbb R^{m+n}$, avoiding the correct, but more cumbersome, $(x,0)$.
$e_i$ will denote the unit vectors in the standard basis,
$\pi_0$ the (oriented) plane $\R^m\times \{0\}$ and $\vec \pi_0$
the $m$-vector $e_1\wedge \ldots \wedge e_m$ orienting it.
We denote by $\p$ and $\p^\perp$ the
orthogonal projections onto, respectively, $\pi_0$ and its orthogonal complement $\pi_0^\perp$. In some cases we need orthogonal projections onto other planes $\pi$
and their orthogonal complements $\pi^\perp$,
for which we use the notation $\p_\pi$ and $\p^\perp_\pi$. 
For what concerns integral currents we use the definitions and the notation of \cite{Sim}.
We isolate the main assumption of our approximation theorem in the following

\begin{ipotesi}\label{ipotesi_base}
For some open cylinder $\bC_{4r} (x)$  (with $r\leq 1$)
and some positive integer $Q$,
\begin{equation}\label{e:(H)}
\p_\sharp T = Q\a{B_{4r} (x)}\quad\mbox{and}\quad
\de T \res \bC_{4r} (x) =0\, .
\end{equation}
\end{ipotesi}

{\Ema The following is the notion of {\em excess}, which represents the main regularity parameter for integral currents.}

\begin{definition}[Excess]\label{d:excess}
For a current $T$ as in Assumption \ref{ipotesi_base} we define the \textit{cylindrical excess} $\bE(T,\bC_{r} (x))$,
the \textit{excess measure} $\e_T$ and its {\em density} $\bd_T$:
\begin{gather*}
\bE(T,\bC_r (x)):= \frac{\|T\| (\bC_r (x))}{\omega_m r^m} - Q ,\\
\e_T (A) := \|T\| (A\times\R^{n})  - Q\,|A| \qquad \text{for every Borel }\;A\subset B_r (x),\\
\bd_T(y) := \limsup_{s\to 0} \frac{\e_T (B_s (y))}{\omega_m\,s^m}= \limsup_{s\to 0} \bE (T,\bC_s (y)),
\end{gather*}
where $\omega_m$ is the measure of the $m$-dimensional unit ball
(the subscripts $_T$ will be omitted if clear from the context).
\end{definition}

The main theorem of the paper is then the following approximation result (for the notation concerning multiple valued functions and their
graphs we refer to \cite{DS1,DS2,DS3}). 

{\Ema \begin{theorem}\label{t:approx lip AM}
There exist constants $M,C_{21}, \beta_0, \varepsilon_{21}>0$ (depending on $m,n, Q$) with the following property. Assume that $T\in \bI_m (\R^{m+n})$ is $\bOmega$-minimal, it satisfies \eqref{e:(H)} in the cylinder $\bC_{4r}(x)$ and $E=\bE(T, \bC_{4r}(x))<\varepsilon_{21}$.
Then, there exist a map $f\colon B_r(x)\to \Iq(\R^n)$ and a closed set 
$K\subset B_r(x)$ such that
\begin{align}
\Lip (f)\leq & C_{21} E^{\beta_0},\label{lip1}
\end{align} 
\begin{equation}\label{lip2}
\bG_f\res(K\times \R^n)=T\res (K\times \R^n)\quad \text{and}\quad |B_r(x)\setminus K|\leq C_{21} E^{\beta_0}\bigl(E+r^2 \bOmega^2 \bigr) r^m\,,
\end{equation}
\begin{equation}\label{lip3}
\Bigl|\norm{T}{}(\bC_r(x))-Q\omega_m r^m -\frac{1}{2}\int_{B_r(x)}|Df|^2\Bigr|\leq C_{21} E^{\beta_0}\bigl(E+r^2 \bOmega^2\bigr)r^m\, ,
\end{equation}
\begin{align}
{\rm osc}(f)\leq C_{21} {\bh}(T, \bC_{4r}(x)).
\end{align}
\end{theorem}

The proof of Theorem \ref{t:approx lip AM} will be achieved in the next three sections. The first one contains the most significant new ideas compared to the approximation of mass minimizing currents as done in \cite{DS3}: here, indeed, we show how to improve upon the almost minimal condition under the assumption that the cylindrical excess is small, thus leading to a refined estimate.
In the two subsequent sections we modify accordingly the computations of \cite{DS3} to prove Theorem~\ref{t:approx lip AM}. Finally, in the last section we show how Theorem~\ref{t:approx lip AM} applies to the currents considered in Definition~\ref{d:semicalibrated} and state for later reference the approximation result which will be used in \cite{DSS3,DSS4} to prove Theorem~\ref{t:finale}.}
From now on constants which depend only upon $m$, $n$ and $Q$ will be called {\em dimensional constants}.

\section{Homotopy lemma}

Before proving the main Lipschitz approximation theorem we need a lemma which estimates carefully the difference in mass between an $\bOmega$-almost minimizer and a competitor in terms of a power of the excess and the constant $\bOmega$.
The key idea is to choose the surface $S$ in \eqref{e:Omega} to be a (suitable perturbation of the) homotopy between
{\Ema two accurately chosen preliminary Lipschitz approximations of $T$ and $R$.
To this regard we introduce the notion of} $E^\beta$-approximation as in \cite[Definition 5.1]{DS3}.  
According to \cite[Proposition 2.2 \& Definition 5.1]{DS3} we then have

\begin{theorem}
There exist dimensional constants $\eps_0, C_{21}>0$ such that, if
$T$ is as in Theorem~\ref{t:approx lip AM} in a cylinder $\bC_{4r} (x)$, $E:=\bE(T,\bC_{4r}(x)) <\eps_0$ and $0<\beta \leq \frac{1}{2m}$, then the following holds. There is a function $u\in \Lip (B_{\sfrac{7r}{2}}(x), \Iqs)$, called {\em $E^\beta$-approximation} of $T$, such that
\[
\Lip (u)\leq C_{21}\,E^{\beta} ,
\]
\[
\bG_u \res (K\times \R^{n})= T\res (K\times \R^{n}),
\]
\begin{equation*}
\mass\big( (T- \bG_u)\res (B_{\sfrac{7r}{2}}(x)\setminus K) \big)+|B_{\sfrac{7r}{2}}(x)\setminus K|\leq C_{21}\,r^m\, E^{1-2\beta} .
\end{equation*}
\end{theorem}

By using the $E^\beta$-approximations we get the following improvement of the $\bOmega$-minimality in the case of small excesses.

\begin{lemma}[Homotopy Lemma]\label{l:homotopy} Let $T$ be an $\bOmega$-almost minimizer which satisfies \eqref{e:(H)}. There are positive dimensional constants $\eps_{22}$ and $C_{25}$ such that,
if $E =\bE(T,\bC_{4r}(x))\leq \eps_{22}$, then the following holds.
For every  $R\in \bI_m (\bC_{3r} (x))$ such that $\de R=\de (T\res \bC_{3r} (x))$, we have
\begin{equation}\label{e:differenza1}
\|T\| (\bC_{3r} (x))\leq \mass(R)+C_{25} r^{m+1} \bOmega E^{\sfrac{1}{2}}\, .
\end{equation}
Moreover, let $\beta\leq \frac{1}{2m}$, $s\in ]r,2r[$, $R = \bG_g\res \bC_s (x)$ for some Lipschitz map $g\colon B_s\to \I{Q}(\R^n)$ with $\Lip (g)\leq 1$ and $f$ be the $E^{\beta}$-approximation of $T$ in $\bC_{3r}$. If $f=g$ on $\partial B_s$ and $P\in \bI_m (\R^{m+n})$ is such that $\partial P =  \partial ((T-\bG_f)\res \bC_s)$, then
\begin{equation}\label{e:differenza2}
\|T\| (\bC_s (x))\leq \mass (\bG_g) + \mass (P) + C_{25} \bOmega \Big(E^{\sfrac{3}{4}} r^{m+1} + \left(\mass (P)\right)^{1+\sfrac{1}{m}} +   \int_{B_s (x)} \cG(f,g)\Big) \, .
\end{equation}
\end{lemma}

\begin{proof} We will first show \eqref{e:differenza1}: in fact \eqref{e:differenza2} follows easily from a portion of the same argument, as it will be highlighted at the end.

Without loss of generality we assume $x=0$. If $\|T\| (\bC_{3r})\leq \mass(R)$ then there is nothing to prove. Hence we can suppose 
\begin{equation}\label{mass}
\mass(R)\leq\|T\| (\bC_{3r}).
\end{equation}  
Define the current $R'\in \bI_m (\bC_{4r})$ by $R':=R+T\res (\bC_{4r}\setminus \bC_{3r})$.
Observe that $\partial (T-R') = 0$. So $\partial (\p_\sharp (T-R')) = 0$. On the other hand 
$\p_\sharp(T-R')=k\a{B_{4r}}$ for some constant $k$ and thus we conclude $\p_\sharp (T-R')=0$.
Therefore $R'$ satisfies \eqref{e:(H)}.
Moreover we notice that, thanks to \eqref{mass}, the cylindrical excess of $R'$ enjoys the following bound:
\[
\bE(R', C_{4r})=\frac{\mass(R')}{\omega_m (4r)^m}-Q\stackrel{(\ref{mass})}{\leq} \frac{\mass(T)}{\omega_m (4r)^m}-Q=\bE(T, C_{4r})=:E.
\]
Let $f,h\colon B_{\sfrac{7r}{2}}\to \Iq(\R^n)$ be the $E^\beta$-Lipschitz approximations of $T$ and $R'$ respectively, in the cylinders $\bC_{\sfrac{7r}{2}}$ {\Ema for some $\beta \in (0, \sfrac{1}{2m}]$}. Then there exist sets $K_T, K_{R'}\subset B_{\sfrac{7r}{2}}(x)$ such that $T\res (K_T\times \R^n) = \bG_f \res (K_T \times \R^n)$ and $R'\res (K_{R'}\times \R^n)= \bG_h \res (K_{R'}\times \R^n)$, fulfilling the following estimates:
\begin{equation}\label{fh1}
\mass((T-\bG_f)\res C_{\sfrac{7r}{2}})\leq C_{21} r^m E^{1-2 \beta}\quad \text{and}\quad\mass((R'-\bG_h)\res C_{\sfrac{7r}{2}})\leq C_{21} r^m E^{1-2 \beta}, 
\end{equation}
\begin{equation}\label{fh2}
|B_{\sfrac{7r}{2}}\setminus K_T|\leq C_{21} r^m E^{1-2\beta}\quad \text{and}\quad|B_{\sfrac{7r}{2}}\setminus K_{R'}|\leq C_{21} r^m E^{1-2\beta},
\end{equation}
\begin{equation}\label{fh3}
\Lip (f)\leq C_{21} E^\beta\quad \text{and}\quad\Lip (h)\leq C_{21} E^\beta.
\end{equation}
Next we set $K:=K_T\cap K_{R'}$ and we notice that by (\ref{fh2})
\begin{equation}\label{fh4}
|B_{\sfrac{7r}{2}} \setminus K|\leq {\Ema C} r^m E^{1-2\beta}.
\end{equation}
Let $|\cdot|$ be the function $|(x,y)|:=|x|_2$ for every $(x,y)\in \R^{m}\times \R^n$, where $|x|_2$ is the Euclidean norm of the vector $x$. By the slicing theory, (\ref{fh1}), (\ref{fh4}) and Fubini's Theorem there exists 
$s\in (3r, 7/2 r)$ such that
\begin{equation}\label{fh5}
\mass(\langle T-\bG_f, |\cdot|, s\rangle) + \mass(\langle R'-\bG_h,|\cdot|,s\rangle)\leq {\Ema C } r^{m-1} E^{1-2 \beta}
\end{equation}
and
\begin{equation}\label{fh6}
|\de B_{s} \setminus K|\leq {\Ema C } r^{m-1} E^{1-2\beta}\, .
\end{equation}
By the Isoperimetric Inequality, there exists $P_T, P_R \in \bI_m(\R^{m+n})$ such that
\[
\partial P_T=\langle T-\bG_f, |\cdot|, s\rangle\qquad \partial P_R = \langle R' - \bG_h, |\cdot|, s \rangle
\]
and
\begin{align*}
\mass(P_T) + \mass (P_R) &\leq C \bigl(\mass(\langle T-\bG_f, |\cdot|, s\rangle\bigr)^{\sfrac{m}{(m-1)}} + C\bigl(\mass(\langle R'-\bG_h, |\cdot|, s\rangle\bigr)^{\sfrac{m}{(m-1)}}\\
&\leq C r^m E^{ m (1-2\beta)/(m-1)}.
\end{align*}
{\Ema Therefore,} we can conclude that
\begin{equation}\label{fh7}
\partial ((T-\bG_f)\res \bC_s) =\de P_T \qquad \partial ((R'-\bG_h)\res \bC_s) = \de P_R
\end{equation}
{\Ema and, since $\beta\leq \frac{1}{2m}$, also}
\begin{equation}\label{fh8}
\mass(P_T) + \mass (P_R) \leq C r^{m} E\, 
\end{equation}
Next consider the functions
\[
f':=\xii\circ f\colon B_{\sfrac{7r}{2}}\to \cQ\subset \R^{N (Q,n)}\quad \text{and}\quad h':=\xii\circ h\colon B_{\sfrac{7r}{2}} \to \cQ\subset \R^{N(Q,n)}\, . 
\]
{\Ema where $\xii:\Iqs \to \R^{N(Q, n)}$ is the bilipschitz embedding in \cite[Section 2.1]{DS1},} and the homotopy between them, defined by
\[
{\Ema \tilde{H} \colon  [0,1] \times B_{\sfrac{7r}{2}} \ni (t,x)}\to (x, tf'(x)+(1-t)h'(x))\in \R^m\times\R^N\,.
\]
Consider the Lipschitz map 
\[
\phi\colon \R^m \times \R^N \ni(x,y)\to (x, \xii^{-1}(\ro( y )))\in \R^m\times \Iq(\R^n)
\]
{\Ema where $\ro:\R^{N(Q, n)} \to \cQ := \xii(\Iqs)$ is the Lipschitz retraction in \cite[Section 2.1]{DS1},} and define $H := \phi \circ \tilde{H}$. $H$ can be seen as a $Q$-valued map $H:   [0,1]\times B_{\sfrac{7r}{2} \to \Iq (\R^{m+n})}$. 
Without changing notation for $H$ we restrict it to $[0,1]\times B_s$ and
following the notation of \cite[Definition 1.3]{DS2} we define $S:= \bT_H$. If we set 
$G := H|_{ [0,1]\times \partial B_s }$ we can
use \cite[Theorem 2.1]{DS2} to conclude that
\begin{equation}\label{fh11}
\de S= (\bG_f-\bG_h)\res \bC_s {\Ema -} \bT_G= (\bG_f - \bG_h)\res \bC_s {\Ema -} P\, ,
\end{equation}
where $P:= \bT_G$.
We now want to estimate $\mass (S)$ and $\mass (P)$ and we will do it using the $Q$-valued area formula in \cite[Lemma 1.9]{DS2}. We start with $\mass (S)$. We fix a point of differentiability $p$ where $DH = \sum \a{DH_i}$. On $[0,1]\times B_s$ we use the coordinates $(t,x)$ and on the target space $\R^{m+n}$ the coordinates $(x,y)$. Let $p = (t_0,x_0)$. It is then obvious that the matrix $DH_i$ can be decomposed as 
\[
DH_i(p)={\Ema\left(\begin{array}{cc}
0_{m\times 1} & I_{m\times m}\\
v_{n\times 1} & A_{n\times m} \, .
\end{array}\right)}
\]
where the matrices $A$ and $v$ can be bound using the following observation. If we consider the map $t\mapsto \Phi (t) := H ( t, x_0)$ and $x\mapsto \Lambda (x) :=  {\Ema t_0 f'(x) + (1-t_0) h'(x)}$, we then have $|v|\leq C \Lip (\Phi)$ and $|A| \leq C \Lip (\Lambda)$, where the constant $C$ depends only on $n$ and $Q$. On the other hand, it is easy to see that $\Lip (\Phi) \leq C \cG (f(x_0), h (x_0))$ and {\Ema $\Lip (\Lambda) \leq C (\Lip (h) + \Lip (f)) \leq E^{\beta}$.} Thus we can estimate
\[
\bJ H_i {\Ema (p)} := \sqrt{\det(DH_i^*{\Ema (p)}\cdot DH_i{\Ema (p)})} \leq C \cG (f(x_0), {\Ema h} (x_0))\, .
\]
Using \cite[Lemma 1.9]{DS2} we then conclude
\[
\mass (S) \leq C \int_{B_s} \cG (f,h)
\]
and, arguing in a similar fashion,
\[
\mass (P) \leq C \int_{\partial B_s} \cG (f,h)\, .
\]
 Observe that $f$ and $h$ coincide, respectively, with the slices of the currents $T$ and $R'$ on any $x_0\in K$. On the other hand, $s>3r $ and $T\res \bC_{4r}\setminus \bC_{3r} = R' \res \bC_{4r}\setminus \bC_{3r}$. We thus conclude that $h = f$ on 
$K\cap \partial B_s$. Let $x\in \partial B_s\setminus K$. By (\ref{fh6}), there exists $x_0\in K\cap \partial B_s$ such that $|x-x_0|\leq Cr E^{(1-2\beta)/(m-1)} = C r E^{2\beta}$ (recall that $\beta {\Ema \leq } \frac{1}{2m}$). Thus
\begin{align*}
\cG (f(x), h(x))&\leq (\Lip( f) + \Lip (h)) \, |x-x_0|\leq C r E^{3 \beta}\, ,
\end{align*}
and so we conclude
\begin{equation}\label{e:molto_piccolo}
\mass (P)\leq C \int_{\partial B_s} \cG (f,h) \leq C r E^{3\beta} |\partial B_s \setminus K | \leq C r^m E^{1+\beta}\leq C r^m E\, .
\end{equation}
On the other hand, we recall that, by a standard variant of the Poincar\'e inequality {\Ema (cf., for example, \cite[4.4.7]{Ziemer})},
\begin{align}
&\int_{B_s} \cG (f, h) \leq C r \|\cG (f,h)\|_{L^1 (\partial B_s)} + C r \|D (\cG (f,h))\|_{L^1 (B_s)}\nonumber\\
\stackrel{\eqref{e:molto_piccolo}}{\leq}& C r^{m+1} E + C r^{1+\sfrac{m}{2}} \left(\int (|Df|^2 + |Dh|^2)\right)^{\sfrac{1}{2}}
\leq C r^{m+1} E^{\sfrac{1}{2}} \, .
\end{align}
Thus,
\begin{equation}\label{fh13}
(\bG_f-\bG_h)\res \bC_s =\partial S + P
\end{equation}
with
\begin{equation}\label{fh14}
\mass(P)\leq C r^m E \quad \text{and}\quad \mass(S)\leq C r^{m+1} E^{\sfrac{1}{2}}.
\end{equation}
Now observe that 
\[
0 = \partial (T-R') =  \partial((\bG_f- \bG_h)\res \bC_s) + \partial (P_T - P_R) = \partial\partial S  +  \partial P + \partial (P_T - P_R)\, .
\]
Hence, $\partial (P+P_T - P_R) =0$ and, by the isoperimetric inequality, there is an $S'$ with $\mass (S') \leq C r^{m+1} E^{1+\sfrac{1}{m}}$ and $\partial S' = P+P_T- P_R$. Additionally, again using the isoperimetric inequality, there are currents $S_T$ and $S_R$ such that 
\begin{align*}
\partial S_T &= (T-\bG_f)\res \bC_s - P_T\\
\partial S_R &= (R'-\bG_h)\res \bC_s - P_R
\end{align*}
and
\begin{align*}
\mass (S_T) \leq C \left(\|T-\bG_f\| (\bC_s) + \mass (P_T)\right)^{\sfrac{(m+1)}{m}} \leq C E^{\sfrac{3}{4}} r^{m+1}\\
\mass (S_R) \leq C  \left(\|{\Ema R' }-\bG_h\| (\bC_s) + \mass (P_R)\right)^{\sfrac{(m+1)}{m}} \leq C E^{\sfrac{3}{4}} r^{m+1}\, .
\end{align*}
In the latter inequalities we have used $\|{\Ema R' }-\bG_h\| (\bC_s) + \|T-\bG_f\| (\bC_s)\leq C E^{1-2\beta}r^m$: in particular
$(1-2\beta) (m+1)/m {\Ema \geq } 1- 1/m^2 \geq 3/4$; observe that this estimate explains the exponent of $E$ in the third summand of the right hand side of \eqref{e:differenza2}.

Thus, setting $S'' = S+S_T-S_R+S'$ we finally achieve
$(T-R') \res \bC_{s} = \partial S''$
and $\mass (S'') \leq C r^{m+1} E^{\sfrac{1}{2}}$. Recalling that $s>3r$ and that $R' = R + T \res (\bC_{4r}\setminus \bC_{3r})$ we conclude
$\partial S'' = (T-R) \res \bC_{3r}$.  
Applying now the $\bOmega$-minimality of $T$ we conclude
\[
\|T\| (\bC_{3r}) \leq \mass (R) + C_{25} r^{m+1} \bOmega E^{\sfrac{1}{2}}\, .
\]
For the proof of \eqref{e:differenza2} we conclude with the same computations, except that this time $f=g$ on $\partial B_s$ and the current $R$ is already given by $\bG_g \res \bC$. The modifications to the argument are then straightforward, given the remark of the previous paragraph. 
\end{proof}

\section{Harmonic approximation and gradient $L^p$ estimates}

In this and in the next section we follow largely \cite{DS3} with minor modifications: on the one hand we have the additional $\bOmega$-error terms, but on the other hand the ambient Riemannian manifold is the {\Ema E}uclidean space. Thus the arguments are somewhat less technical.

\subsection{Harmonic Approximation}
In this subsection we prove that  if $T$ is an almost minimizer then its $E^{\beta}$-Lipschitz approximation is close to a $\D$-minimizing function $w$ {\Ema with estimates which are infinitesimal in the excess.} 

\begin{theorem}[First harmonic approximation]\label{t:o(E)}
For every $\eta_1, \delta>0$ and every $\beta\in (0, \frac{1}{2m})$, there exists a constant $\eps_{23}>0$ with the following property.
Let $T$ be an $\bOmega$-almost minimizer which satisfies Assumption \ref{ipotesi_base} in $\bC_{4{\Ema r_0}}(x)$ .
If $E =\bE(T,\bC_{4{\Ema r_0}}(x))\leq \eps_{23}$ and ${\Ema r_0}\,\bOmega \leq \eps_{23}E^{{\sfrac{1}{2}}}$, then the
$E^{\beta}$-Lipschitz approximation $f$ in $\bC_{3{\Ema r_0}}(x)$ satisfies
\begin{equation}\label{e:few energy}
\int_{B_{2{\Ema r_0}} (x)\setminus K}|Df|^2\leq \eta_1 E\,\omega_m\,(4\,{\Ema r_0})^m = \eta_1\,\e_T(B_{4{\Ema r_0}}(x)).
\end{equation}
Moreover, there exists a $\D$-minimizing function $w$
such that
\begin{gather}
{\Ema r_0}^{-2} \int_{B_{2{\Ema r_0}}(x)}\cG(f,w)^2+
\int_{B_{2{\Ema r_0}}(x)}\big(|Df| - |Dw|\big)^2 \leq \eta_1 E \,\omega_m\,(4\,{\Ema r_0})^m=
\eta_1\, \e_T(B_{4{\Ema r_0}}(x))\, ,\label{e:quasiarm}\\
\int_{B_{2{\Ema r_0}}(x)} |D(\etaa\circ f) - D (\etaa\circ w)|^2 \leq \eta_1 E \,\omega_m\,(4\,{\Ema r_0})^m = \eta_1\, \e_T (B_{4{\Ema r_0}}(x))\, .\label{e:quasiarm_media}
\end{gather}
\end{theorem}

\begin{proof}
{\Ema The proof of the theorem is at all analogous to the one given in \cite[Theorem~3.2]{DS3}: for this reason, we provide here only the principal parts, leaving the details to the readers.}
By rescaling and translating, it is not restrictive to assume that
$x=0$ and $r_0=1$. The proof is by contradiction: assume there exist a constant $c_1>0$, a sequence of positive real numbers $(\eps_l)_l$, a sequence of $\bOmega_l$-minimal currents $(T_l)_{l\in\N}$ and
corresponding $E_l^{\beta}$-Lipschitz approximations $(f_l)_{l\in\N}$ 
such that
\begin{equation}\label{e:contradiction}
E_l:=\bE(T_l,\bC_4) \leq \eps_l\to 0,\;\; \bOmega_l \leq \eps_l E_l^{\sfrac{1}{2}}
\quad\text{and}\quad
\int_{B_2\setminus K_l}|Df_l|^2\geq c_1\, E_l,
\end{equation}
where $K_l:= \{x\in B_3 : \bmax\be_{T_l}(x) < E_l^{2\beta}\}$
{\Ema with $\bmax\be_{T_l}$ denoting the ``non-centered'' maximal function of $\e_{T_l}$:
\begin{equation*}
\bmax\be_{T_l} (y) := \sup_{y \in B_{s} (w)\subset B_{4} (x)} \frac{\be_{T_l} (B_s (w))}{\omega_m\, s^m} = \sup_{y\in B_{s} (w)\subset B_{4} (x)}
\bE ({T_l},\bC_s (w)) .
\end{equation*}
}
Set $\Gamma_l:=\{x\in B_4 : \bmax\be_{T_{l}}(x)\leq 2^{-m}E_l^{2\beta}\}$ and
observe that $\Gamma_l\cap B_3\subset K_l$.
{\Ema From the Lipschitz approximation in \cite[Proposition~3.2]{DS3}}, it follows that
\begin{equation}\label{e:lip(1)}
\Lip (f_l) \leq C_{22} E_l^{\beta},
\end{equation}
\begin{equation}\label{e:lip(2)}
|B_{r} \setminus K_l| \leq C_{22} E_l^{-2\beta} \e_T
\bigl(B_{r+r_0 (l)}\setminus \Gamma_l\bigr) \quad \mbox{for every $r\leq 3$}\, ,
\end{equation}
where $r_0 (l)= 16 \, E_l^{(1-2\beta)/m} <\frac{1}{2}$. 
Then, \eqref{e:contradiction}, \eqref{e:lip(1)} and \eqref{e:lip(2)} give
\begin{equation*}
c_1\, E_l\leq \int_{B_2\setminus K_l}|Df_l|^2\leq
C_{22}\,\e_{T_l}(B_{s}\setminus \Gamma_l)\quad \forall \; s\in\left[\textstyle{\frac{5}{2}},3\right].
\end{equation*}
Setting $c_2:=c_1/(2C_{22})$, we have $2c_2 E_l \leq \be_{T_l} (B_s \setminus \Gamma_l)=
\be_{T_l} (B_s) - \be_{T_l} (B_s \cap \Gamma_l)$, thus leading to
\begin{equation}\label{e:improv}
\e_{T_l}(\Gamma_l\cap B_s)\leq \e_{T_l}(B_s)-2\,c_2\,E_l\, ,
\end{equation}
for $l$ large enough.
Next observe that $\omega_m 4^m E_l = \e_{T_l} (B_4) \geq \e_{T_l} (B_s)${\Ema, because $\be_{T_l}$ is a positive measure under the Assumption \ref{ipotesi_base}.} 
Therefore, by the Taylor expansion in \cite[Corollary 3.3]{DS2},
\eqref{e:improv}
and $E_l\downarrow 0$, it follows that, for every $s\in\left[5/2,3\right]$,
\begin{align}\label{e:improv2}
\int_{\Gamma_l\cap B_s}\frac{|Df_l|^2}{2} & \leq (1+C\,E_l^{2\beta})\,\e_{T_l}(\Gamma_l\cap B_s)\notag\\
&\leq (1+C\,E_l^{2\beta}) \,\Big(\e_{T_l}(B_s)-2\,c_2\,E_l\Big)
\; \leq \e_{T_l}(B_s)-c_2\,E_l.
\end{align}
Our aim is to show that \eqref{e:improv2} contradicts the {\Ema $\bOmega_l$}-almost minimizing property \eqref{e:Omega} of $T_l$.
{\Ema This is shown by constructing a suitable competitor $S_l$ for $T_l$, via a careful modification of the $E_l^{\beta}$-approximations $f_l$.
The construction of the competitor $S_l$ is identical to the one done in \cite[pages 1854-1857]{DS3}, actually simplified by the fact that our currents $T_l$ are supported in $\R^{m+n}$ and not in a Riemannian manifold. Therefore, we omit here the details of the computations (which can be found in full details in the PhD thesis of the third author, \cite{Luca-tesi}) and recall only the conclusion: there exist integer rectifiable currents $S_l$ such that $\de S_l = \de (T_l\res \bC_4)$ and}
\begin{align}
\mass (S_l) - \mass (T_l)
\leq{}&
-\frac{c_2\, E_l}{4}
+  C\,E_l^{1+\gamma}\label{e:differenza_2}\, .
\end{align}
{\Ema Now using} \eqref{e:differenza1} of the Homotopy Lemma \ref{l:homotopy} we have the upper bound
\[
\mass(S_l)-\mass(T_l)\geq-C_{25}\bOmega_l E_l^{\sfrac{1}{2}} \geq -C_{25} \eps_l E_l.
\]
Combining this inequality with \eqref{e:differenza_2} we obtain
\[
\frac{c_2 E_l}{4}\leq CE_l^{1+\gamma}+C\eps_l E_l
\]
which for $E_l, \eps_l$ sufficiently small (and hence for $l$ large enough) provides the desired contradiction.

\medskip

For what concerns \eqref{e:quasiarm}, we argue similarly.
Let $(T_l)_l$ be a sequence with vanishing $E_l:=\bE (T_l, \bC_4)$, contradicting
the second part of the statement and perform the same analysis as before.
Up to subsequences, one of the following statement must be false:
\begin{itemize}
\item[(i)] $\lim_l \int_{B_2} |Dg_l|^2 = \int_{B_2} |Dh_{l_0}|^2$, for any $l_0$
(recall that $\int_{B_2} |Dh_{l}|^2$ is constant);
\item[(ii)] $h_l$ is $\D$-minimizing in $B_2$.
\end{itemize}
If (i) is false, then there is a positive constant $c_2$ such that, for every $r\in [5/2, 3]$,
$$
\int_{B_r} \frac{|Dh_l|^2}{2} \leq \int_{B_r} \frac{|Dg_l|^2}{2} - c_2\leq \frac{\e_{T_l} (B_r)}{E_l}-\frac{c_2}{2},
$$
for $l$ large enough.
Therefore we can argue exactly as in the proof of \eqref{e:few energy} (using $h_l$ instead of $H_l$ to construct the competitors) and
reach a contradiction.
If (ii) is false, then $h_l$ is not $\D$-minimizing in $B_{5/2}$.
{\Ema This implies (cp.~\cite[pages 1857-1859]{DS3})} that we can find a competitor $F_l$ satisfying, for any $r\in [5/2,3]$, 
$$
\int_{B_r} \frac{|DF_l|^2}{2} \leq 
\int_{B_r}\frac{ |Dh_l|^2}{2} -c_2 \leq \lim_l \int_{B_r} \frac{|Dg_l|^2}{2} -2\,c_2\leq
\frac{\e_T(B_r)}{E_l}-\frac{c_2}{2}\, ,
$$
provided $l$ is large enough (where $c_2>0$ is a constant independent of $r$ and $l$). On the other hand, since $F_l = h_l$ on $B_3\setminus B_{5/2}$,
$\|\cG (F_l, g_l)\|_{L^2 (B_3\setminus B_{5/2})} \to 0$ and we argue
as above with $F_l$ in place of $H_l$ and reach a contradiction in this case as well.
{\Ema (The details of this argument are also reported in the PhD thesis of the third author \cite{Luca-tesi}).}
\end{proof}

\subsection{Improved excess estimate.} 
The higher integrability of the Dir-minimizing functions {\Ema (cp.~\cite[Theorem~6.1]{DS3})} and the harmonic
approximation {\Ema in Theorem~\ref{t:o(E)}} lead to the following estimate, which we call ``weak'' since we will improve it in the next section with Theorem~\ref{t:higher}.

\begin{proposition}[Weak excess estimate]\label{p:o(E)}
For every $\eta_{2}>0$, there exist $\eps_{24}, C_{26} >0$ with the following property.
Let $T$ be an $\bOmega$-almost minimizer and assume it satisfies \eqref{e:(H)} in $\bC_{4s} (x)$.
If $E =\bE(T,\bC_{4s}(x))\leq \eps_{24}$, then
\begin{gather}\label{e:o(E)1}
\e_T(A)\leq \eta_{2}\, \be_T(B_{4s}(x))
+ C_{26} \,\bOmega^{2}\,s^{m+2},
\end{gather}
for every $A\subset B_{s}(x)$ Borel with $|A|\leq \eps_{24}|B_{s}(x)|$ ($C_{26}$ depends {\em only} on $\eta_{2}, m,n$ and $Q$).
\end{proposition}

\begin{proof}
{\Ema The proof is a minor modification of \cite[Proposition~6.4]{DS3}: nevertheless, being very short, we provided here a brief account of all the arguments.}

Without loss of generality, we can assume $s=1$ and $x=0$.
We distinguish the two regimes: $\hat{\eps}^2 E \leq \bOmega^2$ and $\bOmega^2 \leq \hat{\eps}^2 E$, where $\hat{\eps} \leq \eps_{24}$ is a 
parameter whose choice will be specified later.
In the former, clearly
$\be_T(A) \leq C\,E \leq C\, \bOmega^{2}$.
In the latter, we let $f$ be the $E^{\sfrac{1}{4m}}$-Lipschitz approximation of $T$ in $\bC_{3}$.
By a Fubini-type argument as the ones already used in the previous sections, we find a radius $r\in (1,2)$ and a current 
$P$ with $\mass (P)\leq C E^{1+\gamma}$ and $\partial ((T-\bG_f)\res \bC_r) = \partial P$ for some $\gamma (m)>0$. We can thus apply the Homotopy
Lemma~\ref{l:homotopy} to $R = \bG_f \res \bC_r + P + T\res (\bC_3 \setminus \bC_r)$:
\begin{align}\label{e:T min}
\|T\| (\bC_r)&\leq \mass(R\res\bC_r) +C\bOmega E^{\sfrac{1}{2}} \nonumber \leq
\|\mathbf{G}_f\| (\bC_r)+C\hat{\eps} E +CE^{1+\gamma}\nonumber\\
&\leq
Q\,|B_r|+\int_{B_r}\frac{|Df|^2}{2}+C \hat{\eps} E + C\,E^{1+\gamma},
\end{align}
for some positive $\gamma$ (possibly smaller than the previous one), where
we used the Taylor expansion in \cite[Corollary 3.3]{DS2}.

On the other hand, using the Taylor expansion for the part of the current which coincides with the graph of $f$, we deduce as well that
\begin{align}\label{e:T below}
\|T\|(\bC_r)&= \|T\| ((B_r\setminus K)\times\R^{n})
+\|T\| ((B_r\cap K)\times\R^{n})\notag\\
&\geq \| T\| ((B_r\setminus K)\times\R^{n})+
Q\,|B_r\cap K|+\int_{B_r\cap K}\frac{|Df|^2}{2}-C\,E^{1+\gamma}.
\end{align}
Subtracting \eqref{e:T below} from \eqref{e:T min}, we then have
\begin{equation}\label{e:out K1}
\e_{T}(B_r\setminus K)\leq \int_{B_r\setminus K}\frac{|Df|^2}{2}+C \hat{\eps} E +C E^{1+\gamma},
\end{equation}
and we recall that the constant $C$ is independent of $\hat\eps$.
{\Ema Therefore, taking into account \eqref{e:few energy} of Theorem~\ref{t:o(E)}, we conclude that the excess on the exceptional set $B_r\setminus K$ is infinitesimal with respect to the $E$ if $\eps_{24}$ is chosen small enough, namely
\begin{equation}\label{e:out K}
\e_{T}(B_r\setminus K)\leq \eta\,E^{},
\end{equation}
for a suitable $\eta>0$.
Let now $A\subset B_1$ be such that $|A|\leq \eps_{24}\,\omega_m$.
Combining \eqref{e:out K} with the Taylor expansion and with \eqref{e:quasiarm} of
Theorem~\ref{t:o(E)}, we have
\begin{equation}\label{e:on A1}
\e_T (A)\leq \e_T (A\setminus K)+\int_A \frac{|Df|^2}{2}+  C\,E^{1+\gamma}
\leq \int_A
\frac{|Dw|^2}{2}+ 2\,\eta\,\be_T(B_{4}),
\end{equation}
where $w$ is a $\D$-minimizing and $\eps_{24}$ is assumed small enough.
Hence, we infer the conclusion \eqref{e:o(E)1} from the higher integrability of the gradient of $\D$-minimizing functions given in \cite[Theorem~6.1]{DS3} (see \cite[page 1861]{DS3} for the simple argument).}
\end{proof}

\subsection{Gradient $L^p$ estimate.} 
{\Ema One of the key points of the proof of Theorem~\ref{t:approx lip AM} is to show an $L^p$ estimate, for some $p>1$, for the density $\bd$ of the excess measure of an $\bOmega$-almost minimizer.}

\begin{theorem}[Gradient $L^p$ estimate]\label{t:higher1}
There exist constants $p_2 >1$ and $C, \eps_{25}>0$ (depending on $n, Q$)
with the following property. Assume $T$ satisfies \eqref{e:(H)} in the cylinder $\bC_4$.
If $T$ is an $\bOmega$-almost minimizer and $E=\bE (T,\bC_4)< \eps_{25}$, then
\begin{equation}\label{e:higher1}
\int_{\{\bd\leq1\}\cap B_2} \bd^{p_2} \leq C\, E^{p_2-1} \left(E + \bOmega^{2}\right).
\end{equation}
\end{theorem}

\begin{proof}
{\Ema The proof is the same as the proof of \cite[Theorem~2.3]{DS3}, where \cite[Proposition~6.4]{DS3} is replaced by our Proposition~\ref{p:o(E)}.}
\end{proof}

\section{Strong excess estimate and proof of Theorem~\ref{t:approx lip AM}}

\subsection{Almgrem's strong excess estimate.} Thanks to the higher integrability of Theorem \ref{t:higher1}, we can control the excess where $\bd\leq 1$. To control it outside this region, we {\Ema prove the following strengthened version of Proposition~\ref{p:o(E)}.}

\begin{theorem}[Almgren's strong excess estimate]\label{t:higher}
There are constants $\eps_{21},\gamma_{2}, C_{27}> 0$ (depending on $n,Q$)
with the following property.
Assume $T$ satisfies Assumption \ref{ipotesi_base} in $\bC_4$ and is $\bOmega$ almost minimizing.
If $E =\bE(T,\bC_4) < \eps_{21}$, then
\begin{equation}\label{e:higher2}
\e_T (A) \leq C_{27}\, \big(E^{\gamma_{2}} + |A|^{\gamma_2}\big) \left(E+\bOmega^2\right)
\quad \text{for every Borel }\; A\subset B_{1}.
\end{equation}
\end{theorem}

\begin{proof}
{\Ema
The proof follows the same scheme in \cite{DS3}.
First of all, by a regularization by convolution technique, we construct a subset of radii $B\subset[1,2]$ with $|B|>\frac12$ with the property that, for every $\sigma \in B$, there exists a $Q$-valued function $g \in \Lip(B_\sigma, \Iqs)$ such that
\begin{gather}
g\vert_{\de B_{\sigma}}=f\vert_{\de B_{\sigma}},\quad
\Lip(g)\leq C_{28}\,E^{\beta_1}\label{e:g1}\\
\int_{B_{\sigma}}|Dg|^2\leq \int_{B_{\sigma}\cap K}|Df|^2 + C_{28}\, E^{\gamma_{3}}\big(E+\bOmega^2 \big),\label{e:g2}
\end{gather}
where $f$ is the $E^{\beta_1}$-Lipschitz approximation of the $\bOmega$-minimal current $T$ and $\gamma_3, C_{28}$ are dimensional positive constants.
The proof of the above estimates is given in \cite[Proposition~7.3]{DS3}.
}

Using now the isoperimetric inequality and a slicing argument, we find a radius $\sigma\in B$ and $P\in \bI_m (\R^{m+n})$ with $\partial P = \partial ((T-\bG_f)\res \bC_s)$ and $\mass (P)\leq C E^{1+\gamma}$. We can therefore apply the Homotopy Lemma~\ref{l:homotopy} to conclude that
\begin{equation}\label{e:vari_pezzi}
\|T\| (\bC_\sigma) \leq \|\bG_g\| (\bC_\sigma) + C \bOmega \int_{B_\sigma} {\Ema \cG} (g,f) + C E^{1+\gamma} {\Ema + C\,\bOmega\,E^{\sfrac{3}{4}}} .
\end{equation}
Then, from \eqref{e:vari_pezzi}, \eqref{e:g2}{\Ema, the inequality
$2\bOmega\,E^{\sfrac{3}{4}} \leq E^\gamma \bOmega^2 + E^{\sfrac{3}{2}-\gamma}$
(for any $\gamma < \sfrac{1}{2}$)}
and the Taylor expansion for $\mass (\bG_g)$ we achieve 
\begin{equation}\label{e:massa1}
\|T\| ( \bC_\sigma) \leq Q\,|B_\sigma|+
\int_{B_\sigma\cap K}\frac{|Df|^2}{2}+ CE^\gamma (E
+ \bOmega^2) + C \bOmega \int_{B_\sigma} \bG (g,f)\,,
\end{equation}
for some $\gamma>0$. On the other hand, by the Taylor's expansion in \cite[Corollary 3.3]{DS2},
\begin{align}\label{e:massa2}
\| T\| (\bC_s)&=\|T\| ((B_s\setminus K)\times \R^{n})+
\|\mathbf{G}_f\| ((B_s\cap K)\times \R^n)\notag\\
&\geq \|T\| ((B_s\setminus K)\times \R^{n})+Q\, |K\cap B_s|+
\int_{K\cap B_s}\frac{|Df|^2}{2}-C\, E^{1+\gamma},
\end{align}
possibly changing the value of $\gamma>0$.
Hence, from \eqref{e:massa1} and \eqref{e:massa2}, we get
\begin{equation}\label{e:eccesso fuori}
\e_T(B_s\setminus K)\leq
C\, E^{\gamma}\,(E+\bOmega^2) + C \bOmega \int_{B_\sigma} \bG (g,f).
\end{equation}

Next note that, by the Taylor expansion of the mass of the graph of $f$, it follows that $|Df|^2 \leq C\,\d_T\leq C E^{2\beta}<1$ a.e.~in $K$: indeed, in all Lebesgue points of $K$ and $|Df|^2$ we have that
\begin{align*}
|Df|^2(x) &=\lim_{s\to0}\frac{\int_{B_s(x)\cap K}|Df|^2}{\omega_m\,s^m} \leq 
C\,\lim_{s\to0} \frac{\e_{\bG_f}(B_s(x)\cap K)}{\omega_m\,s^m}
\\
&\leq C\,\limsup_{s\to0} \frac{\e_{T}(B_s(x))}{\omega_m\,s^m} = C \d_T(x).
\end{align*}

Therefore, for every $A\subset B_1$ Borel set, we can use the higher integrability of 
$|Df|$ in $K$ {\Ema given by Theorem~\ref{t:higher1}} to get
\begin{align*}
\e_T(A)&\leq\e_T(A\cap K)+\e_T( A\setminus K)\\
&\leq
\int_{A\cap K}\frac{|Df|^2}{2}+C\,E^{1+\gamma}+C\, E^{\gamma}\,(E+\bOmega^2) + C \bOmega \int_{B_\sigma} \bG (g,f)\notag\\
&\leq C\,|A\cap K|^{\frac{p_2-1}{p_2}}\left(\int_{A\cap K}|Df|^{q_2}\right)^{\sfrac{2}{q_2}}
+C\, E^{\gamma}\,(E+\bOmega^2)+ C \bOmega \int_{B_\sigma} \bG (g,f)\notag\\
&\leq 
C\, |A|^{\frac{p_2-1}{p_2}}\, \left(E + \bOmega^2\right)+C\, E^{\gamma}\,(E+\bOmega^2)
+ C \bOmega \int_{B_\sigma} \bG (g,f).
\end{align*}

{\Ema In order to conclude the proof we need only to estimate the term $\int_{B_\sigma} \bG (g,f)$.
For this part of the argument it is important to recall the construction of the map $g$ in \cite{DS3}.
We introduce the following notation.
Given two (vector-valued) functions $h_1$ and $h_2$ and two radii $0<s<r$, we denote by $\lin(h_1,h_2)$ the
linear interpolation in $B_{r}\setminus \bar B_s$ between $h_1|_{\partial B_r}$ and $h_2|_{\partial B_s}$, i.e., if $(\theta, t)\in {\mathbb S}^{m-1}\times [0, \infty)$ are spherical coordinates, then
$$
\lin (h_1, h_2) (\theta, t) = \frac{r-t}{r-s}\, h_2 (\theta, s)
+ \frac{t-s}{r-s}\, h_1 (\theta, r)\, .
$$
Next, we fix two parameters $\delta>0$ and $\eps>0$,
radii $1<r_1<r_2<r_3< 2$, given by
\[
r_3 = \sigma, \quad r_2=r_3 -s \quad \text{and}\quad r_1=r_2-s,
\]
with $\sigma \in B$ the radius in the estimates \eqref{e:g1} and \eqref{e:g2} (whose existence is established in \cite{DS3}) and with $\eps=E^a$, $\delta=E^b$ and $s=E^c$, where 
\begin{equation*}
a=\frac{1-2\,{\beta_1}}{2m},\quad b=\frac{1-2\,{\beta_1}}{4m\,(n\,Q+1)}
\quad\textrm{and}\quad c=\frac{1-2\,{\beta_1}}{8^{nQ}\, 4m\,(n\,Q+1)}.
\end{equation*}
Fix also $\ph\in C^\infty_c(B_1)$ a standard nonnegative mollifier.
We set $f' := \xii \circ f$. Recall the Lipschitz maps $\ro$ and $\ro^\star_\delta$ of \cite[Theorem 2.1]{DS1} and \cite[Proposition~7.2]{DS3}, respectively, and
define:
\begin{equation}\label{e:v}
\g:=
\begin{cases}
\sqrt{E}\,\ro \circ \lin\left(\frac{f'}{\sqrt{E}},\ro^\star_\delta\left(\frac{f'}{\sqrt{E}}\right)\right)& \text{in }\; B_{r_3}\setminus B_{r_2},\\
\sqrt{E}\,\ro \circ \lin\left(\ro^\star_\delta\left(\frac{f'}{\sqrt{E}}\right),\ro^\star_\delta
\left(\frac{f'}{\sqrt{E}}*\ph_\eps\right)\right)
& \text{in }\; B_{r_2}\setminus B_{r_1},\\
\sqrt{E}\,\ro^\star_\delta\left(\frac{f'}{\sqrt{E}}*\ph_\eps\right)& \text{in }\; B_{r_1}.
\end{cases}
\end{equation}
Finally set $g:= \xii^{-1}\circ g'$.
}
In particular, {\Ema recalling that $\xii^{-1}$ is Lipschitz continuous and $f= \xii^{-1} \circ f$, we can estimate as follows}
\begin{align*}
\int_{B_\sigma}\cG (f,g)
\leq{}&C \underbrace{\int_{B_\sigma\setminus B_{\sigma -s}}\Bigl|f'-\sqrt{E} \ro\circ \lin\Bigl(\frac{f'}{\sqrt{E}},\ro{\Ema^\star_\delta}\Bigl(\frac{f'}{\sqrt{E}}\Bigr)\Bigr)\Bigr|}_{I_1}+\\
+{}&C \underbrace{\int_{B_{\sigma-s}\setminus B_{\sigma-2s}}\Bigl|f'-\sqrt{E} \ro\circ \lin\Bigl(\ro{\Ema^\star_\delta} \Bigl(\frac{f'}{\sqrt{E}}\Bigr),\ro{\Ema^\star_\delta}\Bigl(\frac{f'}{\sqrt{E}}*\ph_\eps\Bigr)\Bigr)\Bigr|}_{I_2}+\\
+{}&C \underbrace{\int_{B_{{\Ema \sigma-2s}}}\Bigl|f'-\sqrt{E} \ro{\Ema^\star_\delta}\Bigl(\frac{f'}{\sqrt{E}}*\ph_\eps\Bigr)\Bigr|}_{I_3}.
\end{align*} 
We will estimate $I_1,I_2,I_3$ separately.
For what concerns $I_1$, {\Ema we recall that $\ro\circ f'=f'$, $\ro$ is Lipschitz continuous and $\lambda \ro(P)=\ro(\lambda P)$, for every $\lambda>0, P\in \cQ$, since $\cQ$ is a cone; therefore,} 
\begin{align*}
I_1
\leq{}& C \int_{\sigma-s}^\sigma \int_{\partial B_t}\sqrt{E}\, \Bigl|\frac{f'}{\sqrt{E}}-\frac{t+s-\sigma}{s}\frac{f'}{\sqrt{E}}- \frac{\sigma-t}{s}\ro{\Ema^\star_\delta}\Bigl(\frac{f'}{\sqrt{E}}\Bigr)\Bigr|\,dt\\
= {} & C\sqrt{E}\int_{\sigma-s}^{\sigma} \frac{\sigma-t}{s} \int_{\partial B_t}\Bigl|\frac{f'}{\sqrt{E}}-\ro{\Ema^\star_\delta}\Bigl(\frac{f'}{\sqrt{E}}\Bigr)\Bigr|\, dt
\leq{} C \sqrt{E} \delta^{8^{-n Q}}\,|B_\sigma\setminus B_{\sigma-s}|\leq C E^{\sfrac{1}{2}+c}
\end{align*}
where we used {\Ema $|\ro^\star_\delta(P) - P| \leq C\, \delta^{8^{-nQ}}$ from \cite[Proposition~7.2]{DS3}} and $|B_\sigma\setminus B_{\sigma-s}|\leq C s \leq C E^c$.
We next bound $I_2${\Ema : similarly as for $I_1$}
\begin{align*}
I_2
\leq{}& C\sqrt{E}\int_{\sigma-2s}^{\sigma-s} \int_{\partial B_t}\,\Bigl|\frac{f'}{\sqrt{E}}-\frac{t+2s-\sigma}{s}\ro{\Ema^\star_\delta}\Bigl(\frac{f'}{\sqrt{E}}\Bigr)-\frac{\sigma-s-t}{s}\ro{\Ema^\star_\delta}\Bigl(\frac{f'}{\sqrt{E}}*\ph_\eps\Bigr)\Bigr|\\
\leq{}& C \sqrt{E} \int_{\sigma-2s}^{\sigma-s} \int_{\partial B_t}\,\left(\Bigl|\frac{f'}{\sqrt{E}}-\ro{\Ema^\star_\delta}\Bigl(\frac{f'}{\sqrt{E}}\Bigr)\Bigr|+
\frac{\sigma-s-t}{s} \Bigl|\ro{\Ema^\star_\delta}\Bigl(\frac{f'}{\sqrt{E}}\Bigr)-\ro{\Ema^\star_\delta}\Bigl(\frac{f'}{\sqrt{E}}*\ph_\eps\Bigr)\Bigr|\right)\, dt\\
\leq {} & C E^{\sfrac{1}{2}+c}+C \int_{B_{\sigma-s}\setminus B_{\sigma-2s}}\bigl|f'-f'*\ph_\eps\bigr|
\end{align*}
where we have used the fact that $\ro{\Ema^\star_\delta}$ is Lipschitz. The estimate for $I_3$ is similarly given by
\begin{align*}
I_3\leq {}& C \sqrt{E}\int_{B_{\sigma-2s}}\Bigl(\Bigl|\frac{f'}{\sqrt{E}}-\ro{\Ema^\star_\delta}\Bigl(\frac{f'}{\sqrt{E}}\Bigr)\Bigr|+\Bigl|\ro{\Ema^\star_\delta}\Bigl(\frac{f'}{\sqrt{E}}\Bigr)-\ro{\Ema^\star_\delta}\Bigl(\frac{f'}{\sqrt{E}}*\ph_\eps\Bigr)\Bigr|\Bigr)\\
\leq{}&CE^{\sfrac{1}{2}+c} + C \int_{B_{\sigma-2s}} |f'-f'*\ph_\eps|\, . 
\end{align*}
We therefore achieve the estimate
\[
I_2+I_3 \leq C E^{\sfrac{1}{2}+c} + \int_{B_{\sigma-s}} |f'-f'*\ph_\eps|
\]
and to conclude, we compute
\begin{align*}
& \int_{B_{\sigma-s}}\bigl|f'-f'*\ph_\eps\bigr|\leq{}
\int_{B_{\sigma-s}}\int_{B_\eps}\varphi_\eps (x)|f'(y-x)-f'(y)|\,dy\,dx\\
\leq{}&\displaystyle{\int_{B_{\sigma-s}}\int_{B_\eps}\int_0^1\varphi_\eps (x) |Df'(y -tx)\cdot x| \,dt\,dy\,dx}\\
\leq{}& \displaystyle{\int_0^1\int_{B_\eps}\varphi_\eps (x) \eps \int_{B_{\sigma-s}} |Df(y-tx)| \,dy\,dx\,dt}
\leq {}\varepsilon \norm{Df}{L^1(B_\sigma)}
\leq{} C  E^{\sfrac{1}{2}+a}\, ,
\end{align*}
(where we have used the fact that $\eps\leq s$). Putting everything together we conclude that
\[
{\Ema \bOmega \int_{B_\sigma}\cG (f,g)}\leq C\,{\Ema \bOmega}\, E^{\sfrac{1}{2}+\gamma} {\Ema \leq C\,E^\gamma \big(E + \bOmega^2 \big)}
\]
for a suitable $\gamma>0$, thus concluding the proof of the Theorem.
\end{proof}

\subsection{Proof of Theorem \ref{t:approx lip AM}} 
Without loss of generality, we can assume $r=1$ and $x=0$. 
Choose ${\beta_2}<\min \{\frac{1}{2m},\frac{\gamma_3}{2(1+\gamma_3
)}\}$, where $\gamma_3$ is the constant in Theorem~\ref{t:higher}.
Let $f$ be the $E^{\beta_2}$-Lipschitz approximation of $T$.
Clearly \eqref{lip1} follows directly from {\Ema \cite[Proposition~3.2]{DS3}} if ${\Ema \beta_0}<{\beta_2}$.
Set next $A:= \left\{\bmax\be_T> 2^{-m}E^{2\,{\beta_2}}\right\}\cap B_{\sfrac98}$.
By {\Ema \cite[Proposition~3.2]{DS3}}, $|A|\leq C E^{1-2{\beta_2}}$.
Apply estimate \eqref{e:higher2} to $A$
to conclude:
\begin{equation*}
|B_1\setminus K|\leq C\, E^{-2\,{\beta_2}}\,\e_T\left(A\right)
\leq C\, E^{\gamma_3 - 2\beta_2(1+\gamma_3)} (E+\bOmega^2).
\end{equation*}
By our choice of $\gamma_3$ and ${\beta_2}$, this gives \eqref{lip2}
for some positive $\beta_0$. Finally, set $S=\mathbf{G}_f$.
Recalling the strong Almgren's estimate \eqref{e:higher2} and the 
Taylor expansion in \cite[Corollary 3.3]{DS2}, we conclude:
\begin{gather*}
\left| \|T\| (\bC_1) - Q \,\omega_m -\int_{B_1} \frac{|Df|^2}{2}\right|
\leq \e_T(B_1\setminus K)+\e_S (B_1\setminus K)+\left|\e_S (B_1)-\int_{B_1}\frac{|Df|^2}{2}\right|\notag\\
\leq  C\, E^{\gamma_3}(E+\bOmega^2)+C\, |B_1\setminus K|+C\,\Lip(f)^2\int_{B_1}|Df|^2
\leq C\,E^{\gamma_{1}} (E+\bOmega^2).
\end{gather*}
The $L^\infty$ bound follows {\Ema straightforwardly from \cite[Proposition~3.2]{DS3}}. 

{\Ema
\section{Approximation of $2$-dimensional almost minimizing currents}
As mentioned in the introduction, we state here the approximation result for two dimensional currents as in (a), (b) and (c) of Theorem~\ref{t:finale}, which will be used in our subsequent notes \cite{DSS3,DSS4}.
The following are the main assumptions.}

\begin{ipotesi}\label{ipotesi_base2}
In case (a) $\Sigma\subset\R^{m+n}$ is a $C^2$
submanifold of dimension $m + \bar n = m + n - l$, which is the graph of an entire
function $\Psi: \R^{m+\bar n}\to \R^l$ and satisfies the bounds
\begin{equation}\label{e:Sigma}
\|D \Psi\|_0 \leq c_0 \quad \mbox{and}\quad \bA := \|A_\Sigma\|_0
\leq c_0,
\end{equation}
where $c_0$ is a positive (small) dimensional constant. $\omega$ is a $C^1$ $m$-form. 
$T$ is an integral current of dimension $2$ with bounded support. Moreover it satisfies one of the three conditions (a), (b) or (c) in Definition \ref{d:semicalibrated}. In particular in case (a) we have $\supp (T)\subset \Sigma$ and $T$ is area-minimizing in $\Sigma$. In case (b) we assume $\Sigma = \R^{m+n}$ and $T$ is semicalibrated by $\omega$. In case (c) we have that $\Sigma$ coincides with a portion of $\partial \bB_R (p)$, which is the graph of a map $\Psi: \Omega \to \R$ satisfying \eqref{e:Sigma}, for some $\Omega \subset \R^{m+n-1}$.
{\Ema Finally, for some open cylinder $\bC_{4r} (x)$  (with $r\leq 1$)
and some positive integer $Q$, we assume that Assumptions~\ref{ipotesi_base} still holds.}
\end{ipotesi}

\begin{theorem}\label{t: approx lip} There exist constants $M,C_{21}, \beta_0, \varepsilon_{21}>0$ 
(depending on $m,n, \bar{n}, Q$) with the following property. Assume that $T$
satisfies Assumption \ref{ipotesi_base} in the cylinder $\bC_{4r}(x)$ and $E=\bE(T, \bC_{4r}(x))<\varepsilon_{21}$. 
Then, there exist a map $f\colon B_r(x)\to \Iq(\R^n)$, with $\{x\}\times \supp (f(x))\subset \Sigma$ for every $x$, and a closed set 
$K\subset B_r(x)$ such that
\begin{align}
\Lip (f)\leq & C_{21} E^{\beta_0} + C_{21} \bOmega r \quad \text{in case (a) and (c)}\,,\label{lip1-bis}\\
\Lip (f)\leq & C_{21} E^{\beta_0} \qquad\qquad\quad \text{in case (b)}\label{e:lip1_bis}
\end{align} 
\begin{equation}\label{lip2-bis}
\bG_f\res(K\times \R^n)=T\res (K\times \R^n)\quad \text{and}\quad |B_r(x)\setminus K|\leq C_{21} E^{\beta_0}\bigl(E+r^2 \bOmega^2 \bigr) r^m\,,
\end{equation}
\begin{equation}\label{lip3-bis}
\Bigl|\norm{T}{}(\bC_r(x))-Q\omega_m r^m -\frac{1}{2}\int_{B_r(x)}|Df|^2\Bigr|\leq C_{21} E^{\beta_0}\bigl(E+r^2 \bOmega^2\bigr)r^m\, ,
\end{equation}
where $\bOmega = \bA$ in case (a){\Ema, $\bOmega = \|d\omega\|_0$ in case (b) and $\bOmega = \frac{3}{R}$ in case (c).}
If in addition ${\bh}(T, \bC_{4r}(x)):=\sup \{|\p^\perp(x)-\p^\perp(y)|\,:\,x,y\in \supp (T)\cap \bC_{4r}(x)\}\leq r$, then
\begin{align}
&{\rm osc}(f)\leq C_{21} {\bh}(T, \bC_{4r}(x))+C_{21} (E^{1/2}+r\bOmega)r \quad \text{in case (a) and (c)}\label{lip4_(a)}\,,\\
&{\Ema{\rm osc}(f)\leq C_{21} {\bh}(T, \bC_{4r}(x))\quad \text{in case (b).}}\label{lip_(b)(c)}
\end{align}
\end{theorem}

\begin{proof}
{\Ema Case (a) is proved in \cite[Theorem~2.4]{DS3}, while case (b) follows directly from Theorem~\ref{t:approx lip AM} after recalling that semicalibrated currents are $\bOmega$-minimal currents for $\bOmega = \|d\omega\|_0$ by \cite[Proposition~1.2]{DSS1}.

It remains to handle case (c). Again by \cite[Proposition~1.2]{DSS1}, a current satisfying (c) of Definition~\ref{d:semicalibrated} is an $\bOmega$-minimal current for $\bOmega = \frac{3}{R}$. Therefore, we can apply Theorem~\ref{t:approx lip AM}. However, the graph of the map $f$ so obtained is not necessarily contained in $\Sigma$.} We show here how to modify it in such a way to fulfill the requirements of Theorem \ref{t: approx lip}. 
We assume that $\Psi$ is a function whose graph coincides with $\Sigma$ (the connected component of $\partial \bB_R (p) \cap \bC_{4r} (x)$ containing $\supp (T)$) and arguing as in \cite[Remark 1.5]{DS3} we can assume that $\|\Psi_0\|\leq C E^{\sfrac{1}{2}} r + C \bOmega r^2$, $\|D\Psi\|_0\leq C E^{\sfrac{1}{2}} + C \bOmega r$ and $\|D^2\Psi\|_0 \leq C \bOmega$. The domain of $\Psi$ is a subset of $B_{4r}  (x) \times \R^{n-1}$. Let now $f= \sum_i \a{f_i}$ be the function given by Theorem~\ref{t:approx lip AM} and let $\bar f = \sum_i \a{\bar f_i}$, where $\bar f_i (y)$ gives the first $n-1$ coordinates of $f_i (y)$. Observe that on the set $K$ we necessarily have
\[
f (y) = \sum_i \a{(\bar f_i (y), \Psi (y, \bar f_i (y))}\, .
\] 
We then can extend $\bar f$ to $B_r (x) \setminus K$ with $\Lip (\bar f) \leq C \Lip (f)$ and ${\rm osc}\, (\bar f) \leq C {\rm osc}\, (f)$ and hence define $\hat f (y) = \sum_i \a{(\bar f_i (y), \Psi (y, \bar f_i (y))}$ for {\em every} $y\in B_r (x)$ (it must be shown that $(y, \bar f_i (y))$ belongs to the domain of definition of $\Psi$, but this follows easily from the smallness of ${\rm osc}\, (\bar f)$). Obviously $f= \hat f$ on $K$. On the other hand it is straightforward to check that 
\begin{align*}
\Lip (\hat f) \leq & C\, \Lip (\bar f) + C (\Lip (\bar f) +1) \|D\Psi_0\| \leq C E^{\beta_0} + C \bOmega r\\
{\rm osc}\, (\hat f) \leq & C\, {\rm osc} \, (f) + \|\Psi\|_0 \leq C \bh (T, \bC_{4r} (x)) + C (E^{\sfrac{1}{2}} + \bOmega r) r\, .
\end{align*}
In addition we conclude
\[
\left|\int_{B_r (x)} |Df|^2 - \int_{B_r (x)} |D\hat f|^2 \right| \leq (\Lip (f)^2 + \Lip (\hat f)^2) |B_r (x)\setminus K|\leq C {\Ema |B_r (x)\setminus K|\, .}
\]
Thus the estimates in Theorem~\ref{t:approx lip AM} complete the proof.
\end{proof}

\appendix

\section{A singular semicalibrated current}
Here we give an explicit example of a $2$-dimensional current
with a singular point that is semicalibrated by a differential form which is not closed.

Consider a function $u:\R^2 \to \R^2$ which is $C^{\infty}$ but not analytic,
and assume that
\[
u(0) = |\nabla u (0)| = 0.
\]
We need to introduce the following:
\begin{enumerate}
\item $E_1, E_2 \in \cT\big(\gr(u)\big)$ and $E_3, E_4 \in \cN\big(\gr(u)\big)$ 
smooth sections of the tangent and the normal bundles of
$\gr(u) \subset \R^4$ considered as a smooth oriented submanifold such that
\[
E_i(p) \cdot E_j(p) = \delta_{ij} \quad 
\forall\;p \in \gr(u),\;
\forall\; i, j =1,\ldots, 4;
\]
moreover we assume that 
$E_{i}(0) = e_i$ for every $i=1,\ldots, 4$, where $\{e_i\}_{i=1,\ldots,4}$
is the standard basis of $\R^4$;

\item $\theta^1, \ldots, \theta^4$ the dual fields:
\[
\theta^i(p) \big(E_j(p)) = \delta^i_j \quad
\forall\;p \in \gr(u), \;
\forall\; i, j =1,\ldots, 4;
\]

\item $\p_u$ the nearest point projection on $\gr(u)$,
which exists in a tubular neighborhood of the submanifold $\gr(u)$
and therefore, in particular, in $B_{r_0}$ for some $r_0>0$;

\item $\a{D} := (e_3\wedge e_4) \cdot \cH^2\res\{x_1=x_2=0\}$
the current associated to the oriented integration on the vertical
plane $D=\{x_1=x_2=0\}$.
\end{enumerate}

It is now elementary to verify the following claims:
\begin{itemize}
\item[(i)] the smooth $2$-dimensional differential form 
\[
\omega (x) := \theta^1(\p_u(x)) \wedge \theta^2(\p_u(x))+
\theta^3(\p_u(x)) \wedge \theta^4(\p_u(x))
\]
is a semicalibration in $B_{r_0}$;
\item[(ii)] the current $T:=\bG_u + \a{D}$ is semicalibrated by $\omega$ in $B_{r_0}$
and $0 \in \sing(T)$.
\end{itemize}

Note that $\omega$ is not a closed form, for in this case $T$ would be an
area minimizing current thus implying that $\supp(T) \setminus \{0\}$
is locally the graph of an analytic map (cf.~\cite[Theorem~5.5]{Osserman}): this is obviously not the case  
for a generic smooth $u$. 

\medskip

Actually, following the same principles, it is simple
to construct many more examples. In particular it is possible to construct
examples where the semicalibrated current has a branching singularity. However
the corresponding computations are slightly more involved.

\bibliographystyle{plain}
\bibliography{references-App}

\begin{thebibliography}{10}

\bibitem{Alm-memoirs}
F.~J. Almgren, Jr.
\newblock Existence and regularity almost everywhere of solutions to elliptic
  variational problems with constraints.
\newblock {\em Mem. Amer. Math. Soc.}, 4(165):viii+199, 1976.

\bibitem{Alm}
Frederick~J. Almgren, Jr.
\newblock {\em Almgren's big regularity paper}, volume~1 of {\em World
  Scientific Monograph Series in Mathematics}.
\newblock World Scientific Publishing Co. Inc., River Edge, NJ, 2000.

\bibitem{Be3}
Costante Bellettini.
\newblock Almost complex structures and calibrated integral cycles in contact
  5-manifolds.
\newblock {\em Adv. Calc. Var.}, 6(3):339--374, 2013.

\bibitem{BeRi}
Costante Bellettini and Tristan Rivi{\`e}re.
\newblock The regularity of special {L}egendrian integral cycles.
\newblock {\em Ann. Sc. Norm. Super. Pisa Cl. Sci. (5)}, 11(1):61--142, 2012.

\bibitem{Chang}
Sheldon Xu-Dong Chang.
\newblock Two-dimensional area minimizing integral currents are classical
  minimal surfaces.
\newblock {\em J. Amer. Math. Soc.}, 1(4):699--778, 1988.

\bibitem{DS1}
Camillo De~Lellis and Emanuele Spadaro.
\newblock {$Q$}-valued functions revisited.
\newblock {\em Mem. Amer. Math. Soc.}, 211(991):vi+79, 2011.

\bibitem{DS3}
Camillo De~Lellis and Emanuele Spadaro.
\newblock Regularity of area minimizing currents {I}: gradient {$L^p$}
  estimates.
\newblock {\em Geom. Funct. Anal.}, 24(6):1831--1884, 2014.

\bibitem{DS2}
Camillo De~Lellis and Emanuele Spadaro.
\newblock Multiple valued functions and integral currents.
\newblock {\em Ann. Sc. Norm. Super. Pisa Cl. Sci. (5)}, XIV(4):1239--1269,
  2015.

\bibitem{DS4}
Camillo De~Lellis and Emanuele Spadaro.
\newblock Regularity of area-minimizing currents {II}: center manifold.
\newblock {\em Ann. of Math. (2)}, 183(2):499--575, 2016.

\bibitem{DS5}
Camillo De~Lellis and Emanuele Spadaro.
\newblock Regularity of area-minimizing currents {III}: blow-up.
\newblock {\em Ann. of Math. (2)}, 183(2):577--617, 2016.

\bibitem{DSS3}
Camillo De~Lellis, Emanuele Spadaro, and Luca Spolaor.
\newblock Regularity theory for $2$-dimensional almost minimal currents {II}:
  branched center manifold.
\newblock 2015.

\bibitem{DSS4}
Camillo De~Lellis, Emanuele Spadaro, and Luca Spolaor.
\newblock Regularity theory for $2$-dimensional almost minimal currents {III}:
  blowup.
\newblock 2015.

\bibitem{DSS1}
Camillo De~Lellis, Emanuele Spadaro, and Luca Spolaor.
\newblock Uniqueness of tangent cones for $2$-dimensional almost minimizing
  currents.
\newblock 2015.

\bibitem{Grana}
Mariana Gra{\~n}a.
\newblock Flux compactifications in string theory: a comprehensive review.
\newblock {\em Phys. Rep.}, 423(3):91--158, 2006.

\bibitem{GPT}
J.~Gutowski, G.~Papadopoulos, and P.~K. Townsend.
\newblock Supersymmetry and generalized calibrations.
\newblock {\em Phys. Rev. D (3)}, 60(10):106006, 11, 1999.

\bibitem{Gutowski}
Jan Gutowski.
\newblock Generalized calibrations.
\newblock In {\em Progress in string theory and {M}-theory ({C}arg\`ese,
  1999)}, volume 564 of {\em NATO Sci. Ser. C Math. Phys. Sci.}, pages
  343--346. Kluwer Acad. Publ., Dordrecht, 2001.

\bibitem{HL}
Reese Harvey and H.~Blaine Lawson, Jr.
\newblock Calibrated geometries.
\newblock {\em Acta Math.}, 148:47--157, 1982.

\bibitem{Joyce}
Dominic~D. Joyce.
\newblock {\em Riemannian holonomy groups and calibrated geometry}, volume~12
  of {\em Oxford Graduate Texts in Mathematics}.
\newblock Oxford University Press, Oxford, 2007.

\bibitem{Osserman}
Robert Osserman.
\newblock Minimal varieties.
\newblock {\em Bull. Amer. Math. Soc.}, 75:1092--1120, 1969.

\bibitem{PuRi}
David Pumberger and Tristan Rivi{\`e}re.
\newblock Uniqueness of tangent cones for semicalibrated integral 2-cycles.
\newblock {\em Duke Math. J.}, 152(3):441--480, 2010.

\bibitem{RT1}
Tristan Rivi{\`e}re and Gang Tian.
\newblock The singular set of {$J$}-holomorphic maps into projective algebraic
  varieties.
\newblock {\em J. Reine Angew. Math.}, 570:47--87, 2004.

\bibitem{RT2}
Tristan Rivi{\`e}re and Gang Tian.
\newblock The singular set of 1-1 integral currents.
\newblock {\em Ann. of Math. (2)}, 169(3):741--794, 2009.

\bibitem{SS}
Richard Schoen and Leon Simon.
\newblock A new proof of the regularity theorem for rectifiable currents which
  minimize parametric elliptic functionals.
\newblock {\em Indiana Univ. Math. J.}, 31(3):415--434, 1982.

\bibitem{Sim}
Leon Simon.
\newblock {\em Lectures on geometric measure theory}, volume~3 of {\em
  Proceedings of the Centre for Mathematical Analysis, Australian National
  University}.
\newblock Australian National University Centre for Mathematical Analysis,
  Canberra, 1983.

\bibitem{Luca}
Luca {Spolaor}.
\newblock {Almgren's type regularity for Semicalibrated Currents}.
\newblock {\em ArXiv e-prints}, November 2015.

\bibitem{Luca-tesi}
Luca Spolaor.
\newblock {\em Regularity theory for a class of $2$-dimensional almost area
  minimizing currents}.
\newblock PhD Thesis. University of Z\"urich, 2015.

\bibitem{SYZ}
Andrew Strominger, Shing-Tung Yau, and Eric Zaslow.
\newblock Mirror symmetry is {$T$}-duality.
\newblock {\em Nuclear Phys. B}, 479(1-2):243--259, 1996.

\bibitem{Tian}
Gang Tian.
\newblock Gauge theory and calibrated geometry. {I}.
\newblock {\em Ann. of Math. (2)}, 151(1):193--268, 2000.

\bibitem{Wh}
Brian White.
\newblock Tangent cones to two-dimensional area-minimizing integral currents
  are unique.
\newblock {\em Duke Math. J.}, 50(1):143--160, 1983.

\bibitem{Ziemer}
William~P. Ziemer.
\newblock {\em Weakly differentiable functions}, volume 120 of {\em Graduate
  Texts in Mathematics}.
\newblock Springer-Verlag, New York, 1989.
\newblock Sobolev spaces and functions of bounded variation.

\end{thebibliography}

\end{document}